\documentclass[a4paper]{amsart}
\usepackage{amsmath,amssymb}
\usepackage{latexsym}
\usepackage[all]{xy}
\usepackage[latin1]{inputenc}
\usepackage{url}

\usepackage[colorlinks=true]{hyperref}

\newcommand{\ie}{{\sl i.e.}}
\newcommand{\eg}{{\sl e.g.}}
\newcommand{\loccit}{loc. cit.}

\newcommand{\cC}{{\mathcal C}}
\newcommand{\cE}{{\mathcal E}}
\newcommand{\cO}{{\mathcal O}}

\newcommand{\bbZ}{{\mathbb Z}}
\newcommand{\bbP}{{\mathbb P}}
\newcommand{\bbA}{{\mathbb A}}

\newcommand{\Id}{\mathrm{Id}} 

\newcommand{\HHom}[1]{[#1]} 
\newcommand{\TTens}{\otimes} 
\newcommand{\HHomq}[1]{[#1]'} 

\newcommand{\Hom}{{\rm Hom}} 

\newcommand{\W}{{\rm W}} 

\newcommand{\pair}[1]{\{ #1\}} 

\newcommand{\Kos}{{\mathcal K}} 

\newcommand{\one}{{\mathbf 1}}

\newcommand{\dual}{{\sharp}} 
\newcommand{\dualq}{\sharp'} 
\newcommand{\bid}{\varpi} 
\newcommand{\bidq}{\varpi'} 
\newcommand{\q}{{\pi}} 
\newcommand{\fp}{{\alpha}} 
\newcommand{\fh}{{\beta}} 
\newcommand{\fhq}{\beta'} 
\newcommand{\fg}{{\lambda}} 
\newcommand{\ff}{{\mu}} 
\newcommand{\ssp}{{\theta}} 
\newcommand{\rr}{{\zeta}} 
\newcommand{\dd}{{\tau}} 
\newcommand{\eps}{{\varepsilon}} 
\newcommand{\gam}{{\gamma}} 

\newcommand{\Reg}{{\mathcal Reg}} 
\newcommand{\Sch}{{\mathcal Sch}} 
\newcommand{\SmPr}{{\mathcal SmPr}} 

\newcommand{\LTens}{{\otimes^{\rm{L}}}} 
\newcommand{\coh}{Q} 
\newcommand{\Rcoh}{{\rm R}Q} 
\newcommand{\Lf}{{\rm L}f} 
\newcommand{\Rf}{{\rm R}f} 
\newcommand{\LL}{{\rm L}} 
\newcommand{\RR}{{\rm R}} 

\DeclareMathOperator{\RHom}{\mathrm{RHom}} 
\DeclareMathOperator{\SHom}{{\mathcal Hom}} 
\DeclareMathOperator{\Spec}{\mathrm{Spec}}  
\DeclareMathOperator{\Mod}{\mathrm{Mod}}  
\DeclareMathOperator{\Coh}{\mathrm{Coh}}  
\DeclareMathOperator{\Qcoh}{\mathrm{Qcoh}}  
\newcommand{\bVect}[1]{\mathrm{C}_b(\mathrm{Vect}(#1))}  
\newcommand{\bCoh}[1]{\mathrm{C}_b(\mathrm{Coh}(#1))}  

\DeclareMathOperator{\Pic}{\mathrm{Pic}}  

\newcommand{\rel}{\omega} 
\newcommand{\can}{\omega} 

\newcommand{\qis}{qis} 

\newcounter{mydiagram}
\theoremstyle{definition}
\newtheorem{defi}{Definition}[section]

\theoremstyle{plain}
\newtheorem{theo}[defi]{Theorem}
\newtheorem{coro}[defi]{Corollary}
\newtheorem{prop}[defi]{Proposition}
\newtheorem{lemm}[defi]{Lemma}
\newtheorem{nota}[defi]{Notation}

\theoremstyle{remark}
\newtheorem{rema}[defi]{Remark}
\newtheorem{exam}[defi]{Example}

\title{Push-forwards for Witt groups of schemes}
\author{Baptiste Calm{\`e}s and Jens Hornbostel}

\begin{document}
\bibliographystyle{amsplain}

\begin{abstract}
Using suitable closed symmetric monoidal structures on derived categories of schemes, as well as adjunctions of the type $(\Lf^*,\Rf_*)$ and $(\Rf_*,f^!)$ (\ie\ Grothendieck duality theory), we define push-forwards for coherent Witt groups along proper morphisms between separated noetherian schemes. 
We also establish 
fundamental theorems for these push-forwards (\eg\ base
change and projection formula) and provide
some computations.
\end{abstract}

\maketitle

\tableofcontents

\section*{Introduction}
Push-forwards, also known as transfers or norm maps,
exist for many cohomology theories over schemes, \eg\ for
$K$-theory, (higher) Chow groups and algebraic cobordism. They are undoubtedly a 
useful tool for understanding and computing those cohomology theories. 
The present article is about the construction of such push-forward maps for the coherent Witt groups of schemes defined in the seminal work of Balmer\footnote{The modern definition of Witt groups using triangulated categories with
dualities \cite{Balmer00} can be applied either to the derived category
of complexes of locally free sheaves to obtain ``locally free" Witt groups or
to the derived category of complexes with coherent cohomology to obtain
``coherent" Witt groups. As with $K$-theory, it is the latter that is naturally
covariant along proper morphisms, as we prove in this article.
All schemes considered are over $\bbZ[1/2]$ so that the
derived categories involved are $\bbZ[1/2]$-linear and the machinery of
triangular Witt groups applies.}. 
A reader familiar with cohomology theories might think that constructing a
push-forward is probably straightforward. He (or she) should be warned: 
Witt groups are
not an oriented cohomology theory. In particular, push-forwards are, in some sense, only conditionally defined. 
For example,
when $X$ and $Y$ are connected noetherian schemes of finite Krull dimension,
smooth over a field, the Witt groups depend on a line bundle $L$ used to define
the duality and the push-forward takes the form 
(see Theorem \ref{PushForward3_theo})
$$\W^{i+\dim X}(X,\omega_X \otimes f^* L) \to \W^{i+\dim Y}(Y,\omega_Y \otimes L)$$ 
where $\omega$ is the canonical bundle (the highest nontrivial
exterior power of the cotangent
bundle) and $L$ is an arbitrary line bundle on $Y$. In particular, if we pick
a line bundle $K$ over $X$, there is no push-forward starting
from $\W^i(X,K)$ if $K$ is not isomorphic to $\omega_X \otimes f^* L$ for some
$L$ (up to a square $M^{\otimes 2}$, as $\W^i(Y,K) \cong \W^i(Y,K \otimes M^{\otimes 2})$ so only the class of $K$ in 
$\Pic(X)/2$ really matters). 
This fundamental difference with oriented
cohomology theories, where the push-forward is always possible, significantly changes classical computations, as one sees \eg\ in \cite{Gille03},
\cite{Walter03_pre}, \cite{Balmer07_pre}. The groups $\W^{\dim X -i}(X,\omega_X)$ can be considered as a homology theory analogous to a non oriented complex homology theory in topology, but the construction of the push-forward here relies on triangulated monoidal methods. 

\medskip
Besides this article and its precursors on regular schemes \cite{Calmes04_pre} and
\cite{Calmes06_pre},
there are already several articles available on the construction of push-forwards in special cases. In
\cite{Gille03c}, Gille defined push-forwards along finite morphisms in the
affine case. His approach is quite elementary in the sense that he uses direct
computations involving explicit injective resolutions etc. It is useful
to get a hand on concrete forms. In \cite{Nenashev07} and 
\cite{Nenashev09},
Nenashev adapts the oriented cohomology techniques of Panin and Smirnov to the
non-oriented case of locally free Witt groups. He thus obtains push-forwards
along projective morphisms between smooth quasi-projective varieties
over fields. Still another approach using stable ${\bbA}^1$-representability 
of Witt groups can be found in \cite{Hornbostel07}. We understand that there is also some unpublished work of C. Walter
on this subject.
Our approach is different, and it applies to a much larger class of situations: it uses derived functors and Grothendieck duality, so the dualities that appear are canonical and do not depend on choices as the other constructions mentioned above. If necessary, 
choices can be made in order to compare our constructions with others
in the special cases where the latter are defined. 
Fundamental properties such as base change are proved in a simple and conceptual way, and we furthermore obtain the full generality of singular schemes. An example of how those
properties can be used for very concrete computations can be found in the
computation of Balmer and the first author of the Witt group of Grassmann
varieties \cite{Balmer07_pre}. 

\medskip
Let us now explain why we use triangulated {\em monoidal} methods, even though there is no mention of a tensor product in the definition of Witt groups of triangulated categories.
In fact, the proof of many results amounts to verifying that a certain number of diagrams of morphisms of functors such as \eqref{main_diag} below are commutative. 
It might be possible to check this by hand in every concrete situation; 
however, it would be extremely painful: try it for example in the simple case of a regular closed immersion. 
Hence, some kind of systematic method is needed. Our solution to this problem is the use of a convenient setting involving a tensor product, an adjoint internal Hom, functors of the type $f^*$, $f_*$ and $f^!$ and the adjunction relationships between them, that is some variant of the so-called Grothendieck six functors formalism in an arbitrary triangulated category. In this setting, we have shown in \cite{Calmes09} that all the necessary diagrams commute, whereas this article exploits the existence of this structure on various concrete
triangulated categories.
Here is a brief sketch of what is involved: 
Witt groups are defined for triangulated categories $\cC$
equipped with a duality, \ie\ with a contravariant endofunctor $D$ 
on $\cC$ together with a bidual isomorphism of functors $\Id \to D^2=D \circ D$
satisfying $D(\bid_A) \circ \bid_{DA}=\Id_{DA}$ for all objects $A$
of $\cC$.  A morphism between Witt groups is naturally induced by an exact functor $F:\cC_1 \to \cC_2$ (both triangulated categories with dualities resp. $(D_1,\bid_1)$ and $(D_2,\bid_2)$) equipped with an isomorphism of exact functors $\phi:FD_1 \to D_2 F$ which explains how $F$ ``commutes" with the dualities. The fact that the isomorphism $\phi$ is not the identity requires the analysis of its interactions with the other morphisms of functors involved. It is the central problem to solve when proving the main theorems. To start with, this morphism $\phi$ should make the diagram
\begin{equation} \label{dualPres_diag}
\xymatrix{
F \ar[r]^{F\bid_1} \ar[d]_{\bid_2 F} & FD_1 D_1 \ar[d]^{\phi D_1} \\
D_2 D_2 F \ar[r]^{D_2 \phi} & D_2 F D_1
}
\end{equation}
commutative.
In \cite{Calmes09}, we discuss such morphisms of functors and diagrams in
the setting of closed symmetric monoidal categories. More precisely, let
$\cC_1$ and $\cC_2$ be closed symmetric monoidal categories, with tensor
product denoted by $\TTens$ and internal Hom denoted by $\HHom{-,-}$. Given
an object $K$, the functor $D_K:=\HHom{-,K}$ together with the canonical
natural transformation $\bid_K: \Id \to D_K^2$ defines a weak duality functor. Starting with an exact functor $f_*:\cC_1 \to \cC_2$ which has a left adjoint $f^*$ (which is monoidal) and a right adjoint $f^!$, there is a natural transformation
$$\rr: f_*D_{f^!K} \to D_K f_*$$
such that the diagram \eqref{dualPres_diag}, which becomes
\begin{equation} \label{main_diag}
\xymatrix{
f_* \ar[r]^-{f_*\bid_{f^!K}} \ar[d]_{\bid_{K}f_*} & f_*D_{f^!K} D_{f^!K} \ar[d]^{\rr D_{f^!K}} \\
D_K D_K f_* \ar[r]^{D_{f^!K} \rr} & D_{f^!K} f_* D_K
}
\end{equation}
commutes, as shown in \cite{Calmes09}. Therefore, provided $\bid_K$, $\bid_{f^!K}$ and $\rr$ are isomorphisms,
$f_*$ induces a morphism of Witt groups
$$\W(\cC_1,D_{f^!K},\bid_{f^!K}) \to \W(\cC_2,D_K,\bid_K)$$
The present article is a description of how to apply this abstract closed monoidal setting to well-chosen derived categories of schemes, with the derived functors $\Lf^*$, $\Rf_*$ and its right adjoint $f^!$ constructed by Grothendieck duality theory. 

\medskip
The main result of this article is the definition of a push-forward along a proper morphism $f:X \to Y$ of separated noetherian schemes. In its most general form (Theorem \ref{PushForward_theo}), this push-forward is a morphism
$$\xymatrix{\W^i(X,f^! K) \ar[r]^{f_*} & \W^i(Y,K)}$$ 
where $K$ is a dualizing complex on $Y$. This push-forward is induced by the
derived functor $\Rf_*$ and a suitable morphism of functors $\rr_K: \Rf_*
D_{f^! K} \to D_{K} \Rf_*$. This means that a form on a complex $A$ for the
duality $D_{f^! K}$ is sent to a form on the complex $\Rf_* A$ for the duality
$D_K$. We 
further prove that this push-forward respects composition (Theorem \ref{compoPushForward_theo}).

Similarly, for morphisms of finite tor-dimension $f$ we define a pull-back
(Theorem \ref{PullBack_theo}),
that is a morphism
$$\xymatrix{\W^i(Y,K) \to \W^i(X, \Lf^* K)}$$
respecting composition (Theorem \ref{compoPullBack_theo}).

We also prove a flat base change theorem (\ref{baseChange_theo}) relating the
push-forward and the pull-back in cartesian diagrams and a projection formula
in the case of regular schemes (Theorem \ref{ProjFormula_theo}). 
Some explicit computations of transfers are provided in the last section.

\medskip
We assume that schemes are separated and noetherian for the following
technical reasons: quasi-compact and separated are necessary to have an equivalence
between the derived category of quasi-coherent sheaves $D(\Qcoh(X))$ and the
subcategory  $D_{qc}(X)$ of complexes with quasi-coherent homology in the
derived category of all sheaves. Noetherian is used to ensure that the
injectives in $\Qcoh(X)$ 
remain injective in the category of all $\cO_X$-modules. Working without those assumptions would probably require significant improvements in the theory of Grothendieck duality; this is beyond the scope of this article which only intends to apply this theory to Witt groups.

\medskip
Two main cases are discussed. The easier case is when all schemes considered
are regular. Then their derived category $D_{b,c}$ of complexes with coherent and
bounded homology is preserved under the derived tensor product $\LTens$
and under $\RHom$, the derived internal Hom. This endows $D_{b,c}$ with a natural
structure of symmetric monoidal category. The dualizing complexes (see
Definition \ref{dualizing_defi}) are line bundles or shifted line bundles. The
coherent Witt groups are thus defined using the duality $\RHom(-,L)$ for some
line bundle $L$. Furthermore, the derived pull-back $\Lf^*$ for any morphism,
the derived push-forward $\Rf_*$ and its right adjoint $f^!$ for proper
morphisms also preserve $D_{b,c}$. Hence the abstract formalism of
\cite{Calmes09} applies on the nose, and we therefore obtain push-forwards and their classical properties of composition, base change and projection. 

The general case, when schemes are not assumed to be regular, is more complicated. 
Indeed, in this case $\LTens$ or $\RHom$ do not necessarily preserve bounded homology
and so there is no nice closed symmetric monoidal category structure on the category $D_{b,c}$ as the following affine example illustrates. Choose a field $k$ and set $X=\Spec (k[\epsilon]/\epsilon^2)$. Then consider the complex with $k$ concentrated in degree zero. A projective resolution of $k$ is given by
$$\xymatrix{\cdots \ar[r]^{.\epsilon} & k[\epsilon]/\epsilon^2 \ar[r]^{.\epsilon} & k[\epsilon]/\epsilon^2 \ar[r]^{.\epsilon} & k[\epsilon]/\epsilon^2 \ar[r] & k \ar[r] & 0}$$
Thus $k \LTens k$ is the complex
$$\xymatrix{\cdots \ar[r]^0 & k \ar[r]^{0} & k \ar[r]^{0} & k \ar[r] & 0}$$
which has unbounded homology.
On the other hand, the unbounded derived category $D_{qc}$ of complexes with
quasi-coherent cohomology admits a closed symmetric monoidal structure; this
is not completely obvious, see Theorem \ref{ClosedMonStructures_theo}.
But this category is not
suitable to define Witt groups, because there is no obvious (strong) duality
on it and, anyway, as Eilenberg swindle type of arguments show for $K$-theory,
unbounded categories are not the good framework to define cohomology
theories. Still, the closed symmetric monoidal structure on $D_{qc}$ is
useful to prove systematically the commutativity of diagrams such as
\eqref{dualPres_diag}. That is, we can use the framework of
\cite{Calmes09} to prove this commutativity in the large closed symmetric
monoidal category $D_{qc}$ and then notice that all functors used in
the definition of the duality ($\RHom(-,K)$ for some suitable $K$) and the
push-forward ($\Rf_*$) actually restrict to $D_{b,c}$ under mild
additional assumptions. Thus, the commutativity
of the diagrams involved is proved in large categories by
general closed symmetric monoidal methods, but the diagrams actually 
often live in a smaller
category whose Witt groups are interesting. 

A technical point arising is the construction of the functors involved in the symmetric monoidal structure as well as $\Lf^*$, $\Rf_*$ and $f^!$ on the unbounded derived category $D_{qc}$. This relies on the work of Spaltenstein \cite{Spaltenstein88}, the articles of Neeman \cite{Neeman96,Neeman10} and on the very useful notes of Lipman \cite{Lipman09}, which are a reference on Grothendieck duality and contain very detailed explanations of all constructions.

\medskip
The article is organized as follows. In Section \ref{SymmMonCat_sec}, we
recall the closed symmetric monoidal structures of the different categories
we use. In Section \ref{WittGroups_sec}, we use these structures to define
triangulated categories with dualities and related Witt groups. Section
\ref{functors_sec} contains results on the derived functors $\Lf^*$ and
$\Rf_*$ and on Grothendieck duality, \ie\ the construction of the right
adjoint $f^!$ of $\Rf_*$. Section \ref{pushDefi_sec} contains the main result
of the paper, namely Theorem \ref{PushForward_theo}. It explains how to use
\cite{Calmes09} to obtain the definition of push-forwards for the coherent
Witt groups of schemes. It also contains a definition of the finite
tor-dimension (\eg\ flat) pull-back (Theorem \ref{PullBack_theo}). Section
\ref{properties_sec} explains the behavior of the push-forward and the
pull-back under composition, and proves a base change formula relating
them. Section \ref{reformulations_sec} explains possible reformulations of the
push-forward in different contexts and Section \ref{examples_sec} studies
in detail the push-forward in the case of finite field extensions, regular embeddings 
and projective bundles, which is useful for computations and also
allows a comparison with the transfer maps of other authors 
when they are defined. Everything
except some specific computations in the last section works
both for Grothendieck-Witt groups $GW$ and Witt groups $W$.
For simplicity, we stated all results for $W$ only.
\medskip

The present article is a generalization of the main results of
the unpublished preprints \cite{Calmes04_pre} and \cite{Calmes06_pre} 
on regular schemes. 
To keep this article short, some
applications established in \cite{Calmes06_pre}
(d\'evissage/localization, Witt motives and partial results about their
decomposition for cellular varities) are not included here.
Most important, all the abstract theorems
about triangulated symmetric monoidal functors
and adjunctions between them which are crucial for
proving the theorems of this article
are proven in the long article
\cite{Calmes09}.

\medskip
We would like to thank Amnon Neeman for his precise explanations about his approach to dualizing complexes; it enabled us to generalize earlier versions of the results. We would also like to thank Paul Balmer and Bruno Kahn for their constant support, and the referee for his careful reading and detailed comments. 

\section{Closed symmetric monoidal categories} \label{SymmMonCat_sec}

Let $\Sch$ denote the category of separated noetherian schemes
and  $\Reg$ its full subcategory of regular schemes. For any scheme $X$, let $K(X)$ (resp. $D(X)$) denote the homotopy (resp. derived) category of homological complexes of $\cO_X$-modules (without any restriction). We then add subscripts $+$ for bounded above, \ie\ bounded where the differentials go, $-$ for bounded below, $b$ for bounded, \ie\ below and above, $qc$ for quasi-coherent and $c$ for coherent to characterise the derived categories of complexes whose {\sl homology} is as the subscript. For example $D_{b,c}(X)$ is the derived category of complexes of $\cO_X$-modules with coherent and bounded homology, and $D_{qc}(X)$ is the derived category of complexes with quasi-coherent homology. Note that we work with homological notation to be compatible with the literature on Witt groups, but it is easy to switch to cohomological notation by moving bounding subscripts to superscripts and exchanging $+$ and $-$ \ie\ $D_+=D^-$.

For any scheme $X$, the usual tensor product $\TTens$ and internal Hom of
complexes together with the obvious structure morphisms coming from the
corresponding ones for sheaves turn $K(X)$ into a suspended
closed symmetric monoidal category
in the sense of \cite[Section 3]{Calmes09}.
This is completely classical and is detailed in \cite[Appendix]{Calmes09},
where a discussion on sign choices can be found. In particular,
we have a functor $T:K(X) \to K(X)$ given by $(TA)_n=T(A_{n-1})$.

\begin{theo} \label{DXclosedsymm_theo} 
Let $X$ be a scheme.
\begin{enumerate}
\item The tensor product on $K(X)$ admits a left derived functor 
$$\LTens: D(X) \times D(X) \to D(X)$$
together with unit, associativity and symmetry morphisms. 
\item It restricts to
$$D_{qc}(X) \times D_{qc}(X) \to D_{qc}(X).$$
\item When $X \in \Reg$, it furthermore restricts to
$$D_{b,c}(X) \times D_{b,c}(X) \to D_{b,c}(X).$$
\item The internal Hom on $K(X)$ has a right derived functor $\RHom$
$$D(X)^o \times D(X) \to D(X)$$ 
which is a right adjoint to the derived tensor product in the usual special sense (natural in the three variables). 
\item When $X\in \Sch$, $\RHom$ restricts to
$$D_{b,c}(X)^o \times D_{b,c}(X) \to D_{c}(X)$$
as the usual $\RHom$ (computed by replacing the second variable by a quasi-isomorphic complex of injectives in $\Qcoh(X)$). 
\item When $X\in \Reg$, this last restricted $\RHom$ arrives in $D_{b,c}(X)$.
\end{enumerate}
\end{theo}
\begin{proof}
See \cite[Theorem A]{Spaltenstein88} or \cite[2.5.7]{Lipman09} for the
existence of the derived tensor product. It is based on the existence of a
q-flat (also called K-flat) resolution for any complex $C$, \ie\ the existence
of a quasi-isomorphism $Q_C \to C$ where $Q_C$ is a complex such that
$(-\TTens Q_C)$ preserves quasi-isomorphisms. These resolutions can even be
constructed functorially (see \cite[2.5.5]{Lipman09}). The derived tensor
product can then be constructed by taking q-flat resolutions of both variables. The former case is used 
to define the unit morphism and the latter case to define the associativity
and symmetry morphisms directly from the ones of $K(X)$ (see \cite[Theorem
A]{Spaltenstein88} or \cite[2.5.9]{Lipman09}). See \cite[2.5.8]{Lipman09}
for the fact that $\LTens$ restricts to $D_{qc}$. In the regular case, 
by Point (3) of Proposition \ref{equivCat_prop}, we can replace any complex in
$D_{b,c}(X)$ by a bounded complex of locally free sheaves, in which case the 
derived tensor product obviously maps to $D_{b,c}(X)$.

Similarly, the derived internal Hom is constructed using q-injective (also
called K-injective) resolutions: see \cite[Section 1]{Spaltenstein88} for the
definition of a q-injective complex and \cite[Theorem A]{Spaltenstein88} or
\cite[2.4.5]{Lipman09} for the existence of $\RHom$. Adjointness is also
stated in \cite[Theorem A]{Spaltenstein88} (see also \cite[2.6.1]{Lipman09} for more details). 

We now consider $\RHom(A,B)$ with $A,B \in D_{b,c}(X)$ for $X \in \Sch$. 
By Corollary \ref{qinjBounded_coro}, the right derived functor $\RHom$ here is computed as the one in \cite[Prop. II.3.3]{Hartshorne66}. This proves point (5). In the regular case, 
we can compute $\RHom$ by a locally free resolution of the first variable and
then, up to isomorphism, also replace the second variable by a complex of
locally free sheaves. As explained above, both these complexes can be chosen
to be bounded, and since $\SHom(A,B)$ is coherent when $A$ and $B$ are
\cite[5.3.5]{EGA1}, this poves Point (6). 
\end{proof}

Now the subtle point is that $\RHom(M,N)$ is not necessarily an object in $D_{qc}(X)$ when $M$ and $N$ are. To fix this, we use the {\it quasicoherator} $\coh:\Mod(X) \to \Qcoh(X)$ as introduced in \cite[Lemma 3.2 p.187]{SGA6}, which is right adjoint to the inclusion $\Qcoh(X) \subset \Mod(X)$. On an affine space $X=\Spec(A)$, it takes a sheaf of $\cO_X$-modules to the quasi-coherent sheaf associated to the $A$-module of its global sections 
by the tilde construction. Its right derived functor is denoted by $\Rcoh$, as
considered in \cite[Remark 0.4]{AlonsoTarrio97}, 
\cite[Exercises 4.2.3]{Lipman09} or \cite[B.16]{Thomason90}. 
It is a right adjoint to the inclusion $D_{qc}(X) \subseteq D(X)$, and $A \cong \Rcoh(A)$ when $A \in D_{-,qc}(X)$ (in particular for $A \in D_{b,c}(X)$) by \cite[Exp. II, Prop. 3.5.2]{SGA6}. 
An alternative construction of
$\Rcoh$ can be obtained from \cite[Theorem 4.1]{Neeman96}.

\begin{theo}\label{ClosedMonStructures_theo}
For any scheme $X$, 
\begin{enumerate} 
\item the derived tensor product $\LTens$ together with the obvious
  morphisms turns $D(X)$ into a symmetric monoidal category, closed by the
  $\RHom$, and suspended in the sense of \cite[Section 3]{Calmes09}.
\item If $X\in \Sch$, the functor 
$$\Rcoh \circ \RHom: D_{qc}(X)^o \times D_{qc}(X) \to D_{qc}(X)$$ 
is a right adjoint (in the usual special way, see \cite[(v) p. 97]{Kelly71})
to the
restricted tensor product $\LTens$ on $D_{qc}$. This turns $D_{qc}(X)$ into a
suspended closed symmetric monoidal category.
\item If $X\in \Reg$, the usual $\RHom$ is a right adjoint 
(in the usual special way)
 to the restricted tensor product $\LTens$ on $D_{b,c}$. This turns
 $D_{b,c}(X)$ into a suspended closed symmetric monoidal category.
\end{enumerate}
\end{theo}
\begin{proof}
The closed symmetric monoidal structure on $D(X)$ easily follows from Theorem
\ref{DXclosedsymm_theo}. The fact that it is suspended follows, as explained
in \cite[Section 3]{Calmes09}, from the suspended bifunctor structure of
$\RHom$. The symmetric monoidal structure on $D_{qc}(X)$ simply follows from
the fact that $\LTens$ restricts to it. The fact that it is closed is a formal
consequence of the fact that $D(X)$ is closed and that $\Rcoh$ is a right adjoint to the (monoidal) inclusion $\iota:D_{qc}(X) \subseteq D(X)$: 
$$\Hom_{qc}(A \LTens B, C) = \Hom(\iota(A \LTens B),\iota C) \simeq \Hom(\iota A \LTens \iota B, \iota C)$$
$$\simeq \Hom(\iota A, \RHom(\iota B, \iota C))
\simeq \Hom_{qc}(A, \Rcoh \RHom( \iota A, \iota B)).$$
(The closedness - that is the existence of the right
adjoint to the derived tensor product - 
can also be deduced from Brown representability, in the spirit of the examples following \cite[Theorem 4.1]{Neeman96}.)
Point (3) follows from the same considerations, using Theorem \ref{DXclosedsymm_theo} (3) and (6).
\end{proof}

\begin{nota}\label{nothom}
To shorten the notation, let $\HHom{-,-}$ denote the functor $\RHom$, right adjoint to the tensor product on the derived category $D$ and let $\HHomq{-,-}$ denote the functor $\Rcoh \circ \RHom$, right adjoint to the tensor product on the derived category $D_{qc}$.
\end{nota}

Since the derived quasi-coherator is the identity on $D_{-,qc}$
(see above), if $\HHom{A,B} \in D_{-,qc}$ then $\HHom{A,B} \cong \HHomq{A,B}$.

\medskip

We finish this section by pointing out a comment of Neeman: exploiting the fact that for $X \in \Sch$, there are enough flat objects in $D_{qc}$ and his representability result more extensively gives an alternative approach for constructing a closed symmetric monoidal structure on $D_{qc}(X)$ directly without passing through $D(X)$.
 
\section{Witt groups} \label{WittGroups_sec}

From now on, we assume that all schemes
are defined over $\bbZ[1/2]$.

\medskip
To define a Witt group, we need a strong duality on a triangulated
category. Using the previous framework of triangulated closed symmetric
monoidal categories, we recall how $\HHom{-,K}$ and $\HHomq{-,K}$ define dualities. The purpose of this section is to compare the restrictions of these dualities to the subcategory $D_{b,c}$ and to discuss when these dualities are strong on it. For any object $K$, let $\dual_K$ (resp. $\dualq_K$) denote the contravariant exact functor $\HHom{-,K}$ (resp. $\HHomq{-,K}$).

Following \cite[Section 3.2]{Calmes09}, applied to the closed symmetric monoidal
 structure on $D(X)$ with $X$ an arbitrary scheme, we may define the 
{\it bidual morphism}
$$\bid_K: \Id \to \dual_K \dual_K$$
as a morphism of triangulated endofunctors of $D(X)$. From
\cite[Cor. 3.2]{Calmes09}, we obtain that $(D(X),\dual_K,\bid_K)$ is a
triangulated category with weak duality (in the sense of
\cite[Definition 2.1.1]{Calmes09}, so $\bid_K$ is not necessarily
an isomorphism). Similarly, when $X \in \Sch$, we obtain a triangulated category with weak duality $(D_{qc}(X), \dualq_K, \bidq_K)$.

\begin{defi} \label{dualizing_defi}
Let $K$ be an object of $D_{qc}(X)$. It is a {\em dualizing complex} (or it is {\em dualizing}) if
\begin{itemize}
\item[-] the functor $\HHom{-,K}$ preserves $D_{b,c}(X)$ and 
\item[-] the bidual morphism $\bid_K$ is an isomorphism on $D_{b,c}(X)$. 
\end{itemize}
If furthermore it has finite injective dimension, \ie\ it is quasi-isomorphic
to a finite complex of injectives, we say it is an {\em injectively bounded dualizing complex}.
\end{defi}
In the terminology of \cite[Definition 2.1.1]{Calmes09}, the second
condition says that $\dual_K$ is a strong duality on the subcategory $D_{b,c}(X)$.

Note that for any $X \in \Sch$, a dualizing complex $K$ is automatically in $D_{b,c}(X)$ since the natural morphism $K \to \HHom{\cO_X,K}$ coming from the monoidal structure is an isomorphism and $\cO_X$ is coherent. In particular, our definition is exactly the ``modern" \cite[Definition 3.1]{Neeman10}, by Lemma 3.5 of \loccit\ 
Also note that our injectively bounded dualizing complexes are the ``old" dualizing complexes of \cite[V. §2]{Hartshorne66}
\begin{prop} \label{compareDualities_prop}
Let $X \in \Sch$ and $K \in D_{qc}(X)$ be a dualizing complex. Then the functors $\dual_K$ and $\dualq_K$ coincide and the bidual morphisms $\bid_K$ and $\bid'_K$ are equal on the subcategory $D_{b,c}(X)$. 
\end{prop}
\begin{proof}
Since $\HHom{A,K}\in D_{b,c}(X)$ for any $A\in D_{b,c}(X)$, we have
$\HHomq{A,K} \cong \HHom{A,K}$ by the remark after Notation \ref{nothom}
which proves that $\dualq_K \cong \dual_K$. The
bidual morphisms are then equal by the large commutative diagram considered
in the proof of \cite[Theorem 4.1.2]{Calmes09}, in which the $f^*$ should be replaced by the inclusion $D_{qc}(X) \subset D(X)$, which is monoidal by definition of the tensor product on $D_{qc}(X)$. 
\end{proof}

\begin{exam}\label{examplesdual}
\begin{enumerate}
\item A dualizing complex tensored by a shifted line bundle is still a
  dualizing complex. In fact, this is the only freedom: by \cite[Lemma
  3.9]{Neeman10} (see also \cite[Theorem V.3.1]{Hartshorne66} for 
the injectively bounded case),
a dualizing complex is unique up to tensoring by shifted line bundles (the shift can be different on different connected component of $X$).
\item On a Gorenstein scheme $X$ (\eg\ regular), $\cO_X$ itself is dualizing, so by the previous point, the only dualizing complexes are the shifted line bundles. 
\end{enumerate}
\end{exam}

Note that on a regular scheme, the category $D_{b,c}(X)$ itself is closed symmetric monoidal. It follows that dualizing complexes are dualizing objects in the sense of \cite[Definition 3.2.2]{Calmes09} in the category $D_{b,c}(X)$, for $X \in \Sch$.

\begin{theo}
Let $X \in \Sch$ and $K$ be dualizing. Then $(D_{b,c}(X),\dual_K,\bid_K)$ is a
triangulated category with strong duality in the sense of
\cite[Def. 2.1.1]{Calmes09}. Let it be denoted by $\cC_K$ and its Witt groups 
\cite[Definition 2.1.5]{Calmes09} by $\W^i(X,K)$, $i\in \bbZ$.
\end{theo}
\begin{proof}
The functor $\dual_K=\dualq_K$ is a contravariant endofunctor of $D_{b,c}(X)$
and $\bid_K=\bidq_K$ is an isomorphism on this category by definition of
dualizing complexes by Proposition \ref{compareDualities_prop}. The necessary commutative diagrams that $\dualq_K$ and $\bidq_K$ must satisfy simply follow from the fact that they are already satisfied in $D_{qc}(X)$ since $(D_{qc}(X),\dualq_K,\bidq_K)$ is a triangulated category with weak duality. 
\end{proof}

We may thus think of the triangulated category with duality $(D_{b,c}(X),\dual_K,\bid_K)$ as being restricted from $(D(X),\dual_K,\bid_K)$ or from $(D_{qc}(X),\dualq_K,\bidq_K)$, both structures coinciding on $D_{b,c}(X)$.

\begin{rema}
In \cite{Balmer00}, all dualities considered are strict, \ie\ they strictly commute with the suspension, but this assumption is only there for simplicity.
Instead, in \cite[Def. 2.1.1]{Calmes09}, we only assume commutativity up to a natural isomorphism, and all theorems in \cite{Balmer00} are still true
in this more general situation. 
\end{rema}

\begin{rema} Recall (see e.g. \cite[Def. 10.5.1]{Weibel94})
that for a left exact functor $f$ between exact categories, the right derived functor really
is a couple $(\Rf,s)$ with $s:qf \to (\Rf) q$ and $q$ the morphism from the
homotopy category to the derived category. It is only the couple
$(\Rf,s)$ which is unique up to unique isomorphism and therefore
deserves being called {\it the} right derived functor, despite the standard abbreviated notation $\Rf$.
Consequently, the various derived functors, for example $\RHom(-,K)$ (used to define the duality) and $\Rf_*$ (used below to define the push-forward) together with all the morphisms of functors defining the symmetric monoidal structure can be considered as abstract exact functors and morphisms of exact functors. With them, it is possible to define coherent Witt groups and push-forwards by the methods discussed in this article, since these methods only involve the abstract triangulated categories and functors, \ie\ the framework of \cite{Calmes09}. But as such, there is no uniqueness of all these constructions. It is only if we keep as extra data all the structural morphisms of the derived functors (the $s$ part of the couples), and thus the relationship between the closed symmetric monoidal structure on $K(X)$ and the one on $D(X)$, that the whole derived construction becomes unique up to unique isomorphism, thus in particular the induced dualities, pull-backs and push-forwards. 
\end{rema}

\section{The functors $\Lf^*$, $\Rf_*$ and $f^!$} \label{functors_sec}

We now introduce the functors $\Lf^*$, $\Rf_*$ and $f^!$ associated to a morphism of schemes $f$ and explain how they behave with respect to the monoidal structures. The first two
functors are derived functors, whereas the third one is right adjoint to
$\Rf_*$ at the level of derived categories, but is not the derived functor of some underlying
functor on the category of $\cO_X$-modules. The construction of $f^!$ is the
heart of Grothendieck duality theory, for which we refer the reader to
\cite{Hartshorne66}, \cite{Verdier69}, \cite{Neeman96}, \cite{Conrad00} 
or \cite{Lipman09}.

\begin{prop}
Let $f:X \to Y$ be a morphism of schemes. 
\begin{enumerate}
\item The functor $f^*$ admits a left derived functor $\Lf^*: D(Y) \to D(X)$ which restricts to $D_{qc}(Y) \to D_{qc}(X)$.
\item If $f$ is of finite tor-dimension (see \eg\ \cite[Examples (2.7.6)]{Lipman09}) or if $X,Y \in \Reg$, then $\Lf^*$ resticts to $D_{b,c}(Y) \to D_{b,c}(X)$.
\end{enumerate}
\end{prop}
\begin{proof}
For existence, see \cite[Theorem A (iii) or Prop. 6.7]{Spaltenstein88} or
\cite[Example 2.7.3]{Lipman09}. For the fact that it restricts to $D_{qc}$,
see \cite[3.9.1]{Lipman09}. It restricts to $D_{b,c}$ in the finite
tor-dimension case because $\Lf^*$ is then bounded and it respects the
coherence of the cohomology by \cite[Proposition II.4.4]{Hartshorne66}, bearing in mind 
Proposition \ref{qflatBounded_prop}. The case $X,Y \in \Reg$ follows from
Point (3) of Proposition \ref{equivCat_prop}
and Proposition \ref{qflatBounded_prop}.
\end{proof}

\begin{prop}
The usual isomorphism $f^*(A \otimes B) \to f^* A \otimes f^* B$ induces an isomorphism of triangulated bifunctors (in the sense of \cite[Def. 1.4.14]{Calmes09})
$$\fp: \Lf^*(- \TTens -) \to \Lf^* (-) \TTens \Lf^* (-)$$
which turns $\Lf^*$ into a suspended symmetric monoidal functor in the sense of \cite[Section 4]{Calmes09}.
\end{prop}
\begin{proof}
See \cite[Prop. 6.8]{Spaltenstein88}. The morphism $\fp$ is defined as the corresponding one on $K(X)$ after having replaced both variables by q-flat resolutions. It is already an isomorphism on $K(X)$. The commutative diagrams required (compatibility with the associativity, unit and symmetry of the monoidal structures) then easily follow from the corresponding ones on $K(X)$, using Proposition \ref{qflatinj_prop}, Points (1) and (3). 
\end{proof}

By \cite[Proposition 4.1.1]{Calmes09} applied to the symmetric monoidal structure and $\Lf^*$ on $D(X)$, there is a natural morphism 
$$\fh: \Lf^* \HHom{-,-} \to \HHom{\Lf^*(-),\Lf^*(-)}.$$
We also obtain a morphism
$$\fhq:\Lf^* \HHomq{-,-} \to \HHomq{\Lf^*(-),\Lf^*(-)}.$$
using $D_{qc}(X)$ instead of $D(X)$. 

\begin{prop} \label{compareBetas_prop}
Let $X,Y \in \Sch$ and $A,B \in D_{qc}(Y)$. Assuming $\HHom{A,B} \in D_{qc}(Y)$ and $\HHom{\Lf^*A,\Lf^*B} \in D_{qc}(X)$, the morphisms $\fh_{A,B}$ and $\fhq_{A,B}$ coincide.
In particular, when $K$ and $\Lf^* K$ are dualizing, $\fh_K$ and $\fhq_K$ coincide. 
\end{prop}
\begin{proof}
This follows from the commutative diagram
$$\xymatrix{
\Lf^* \HHomq{A,B} \ar[r]^-{\fhq} \ar[d] & \HHomq{\Lf^*A,\Lf^*B} \ar[d] \\
\Lf^* \HHom{A,B} \ar[r]^-{\fh} & \HHom{\Lf^*A,\Lf^*B}
}$$
in which the vertical maps become identities under the assumptions. This diagram is formally obtained from the definitions of $\fh$ and $\fhq$ out of the closed monoidal structures.
\end{proof}

\begin{prop} \label{betaIsoFiniteTorDim_prop}
When $X,Y \in \Sch$ and $f:X \to Y$ is of finite tor-dimension or when $X,Y \in \Reg$ and for any $f: X \to Y$, the natural morphism $\fh$ is an isomorphism on objects in $D_{b,c}$. 
\end{prop}
\begin{proof}
This follows from \cite[Proposition 4.6.6]{Lipman09} for $f$ of finite tor-dimension, the first variable coherent and the second in $D_-$, so in particular for both in $D_{b,c}$. Note that the $\rho$ of \loccit\ coincides with our $\fh$ by definition (compare \cite[(3.5.4.5)]{Lipman09} and \cite[Proposition 4.1.1]{Calmes09}). In the regular case, by Point (3) of Proposition \ref{equivCat_prop}, we can assume our objects are bounded complexes of locally free sheaves, in which case the result follows from \cite[Proposition 4.6.7]{Lipman09}.
\end{proof}

\begin{prop}
Let $f:X \to Y$ be a morphism of schemes. 
\begin{enumerate}
\item The functor $f_*$ admits a right derived functor $\Rf_*:D(X) \to D(Y)$. 
\item The functor $\Rf_*$ restricts to $D_{qc}(X) \to D_{qc}(Y)$ when $f$ is quasi-compact and 
separated, in particular if $X$ and $Y$ are in $\Sch$, see
\cite[Cor. 6.1.10]{EGA1Springer}. 
\item The functor $\Rf_*$ restricts to $D_{b,c}(X) \to D_{b,c}(Y)$ when $f$ is proper and $Y$ is quasi-compact. 
\end{enumerate}
\end{prop}
\begin{proof}
For existence, see \cite[Theorem A (iii)]{Spaltenstein88} or \cite[Examples 2.7.3]{Lipman09}. For the fact that it restricts to $D_{qc}(-)$
see \cite[3.9.2]{Lipman09}.
In the proper case with $Y$ quasi-compact, it restricts to $D_{b,qc}(-)$ by \cite[3.9.2]{Lipman09} and it then preserves coherence of the cohomology \cite[Theorem 3.2.1]{EGA3-2}. Note that we use \cite[Definition 5.3.1]{EGA1} to define coherent modules on non necessarily noetherian schemes. 
\end{proof}

\begin{prop}
For any morphism $f$ of schemes, the functor $\Rf_*$ is a right adjoint to $\Lf^*$
on $D(-)$ and consequently on all full subcategories to which both
functors restrict.
\end{prop}
\begin{proof}
See \cite[Theorem A (iii)]{Spaltenstein88} or \cite[Proposition 3.2.1]{Lipman09}. 
\end{proof}

By \cite[Proposition 4.2.5]{Calmes09} applied to the monoidal structure and the functors on the categories $D(-)$,
we obtain the projection morphism
$$\q: \Rf_*(-) \TTens - \to \Rf_*(- \TTens \Lf^*(-))$$

\begin{theo} \label{projMorphIso_theo}
Let $f:X \to Y$ be quasi-compact and quasi-separated \eg\ proper. 
Then the projection morphism $\q$ is an isomorphism on $D_{qc}$.
\end{theo}
\begin{proof}
This is \cite[Proposition 3.9.4]{Lipman09}. 
\end{proof}

\begin{theo}
For any separated morphism $f:X \to Y$ with $X$ and $Y$ separated and quasi-compact, the functor $\Rf_*:D_{qc}(X) \to D_{qc}(Y)$ has a 
right adjoint $f^!$.
\end{theo}
\begin{proof}
See \cite[Example 4.2]{Neeman96} and use that $D_{qc}(-)$ and $D(\Qcoh(-))$ are
equivalent for separated quasi-compact schemes by Proposition \ref{equivCat_prop} (1).
\end{proof}

\begin{prop} \label{properDualizing_prop}
Let $f:X \to Y$ be a proper morphism of separated noetherian schemes and let
$K$ be a dualizing complex on $Y$. Then $f^! K$ is a dualizing complex on
$X$. If $K$ is an injectively bounded dualizing complex
\ie\ dualizing in the sense of \cite[V. §2]{Hartshorne66}, then $f^! K$ is 
injectively bounded too.
\end{prop}
\begin{proof}
For the case of injectively bounded complexes, 
see \cite[V, §8]{Hartshorne66} or
\cite[Corollary 3]{Verdier69}. For the general case, we reproduce here a proof
of Neeman.
Since the question of whether $f^!K$ is dualizing is local on $X$, we
may assume $Y$ is affine and restrict to an affine open set $U$ of $X$.
As $f$ is of finite type, we have a factorization $U \to \bbA^n \times Y
\to Y$ for some $n$ where the left arrow is a closed embedding.
Taking the closure of $U$ in $\bbP^n \times Y$, we see that
$U$ can be embedded as an open subset of a closed subset of some $Y
\times \bbP^n$.
Hence we only have to show that
closed immersions, open immersions and projections $Y \times \bbP^n \to Y$
respect dualizing complexes. The case of closed immersions is done in
\cite[Theorem 3.14, Remark 3.17 and Lemma 3.18]{Neeman10}; closed
immersions are finite. The case of open immersions is \cite[Theorem
3.12]{Neeman10}.
For projective morphisms $f:{\bbP}^n_Y \to Y$,
one uses that $\RHom(A,f^!K) \cong \RHom(A,f^!\cO \otimes f^*K)
\cong \RHom(A,f^*K) \otimes f^!\cO$,
using Lemma \ref{sspIso_lemm} below and
that $f^!\cO$ is a shifted line bundle by \cite[Section VII.4]{Hartshorne66}.
Then one checks the conditions of Definition \ref{dualizing_defi}
on objects of the form $f^*B$ and $\cO(i)$
which by a theorem of Beilinson \cite{Beilinson78}
generate $D_{b,c}({\bbP}^n_Y)$ as a thick triangulated
category.
\end{proof}

\section{Pull-back and push-forward for Witt groups} \label{pushDefi_sec}

We can now state the main result of this article: the definition of the push-forward for coherent Witt groups along proper morphisms (Theorem \ref{PushForward_theo}). This section also contains a definition of the pull-backs for morphisms of finite tor-dimension (Theorem \ref{PullBack_theo}).

\medskip
Let $f: X \to Y$ be a morphism of schemes. By \cite[Theorem 4.1.2]{Calmes09} applied to the monoidal categories $D(-)$, $\fh_K: \Lf^* \dual_K \to \dual_{\Lf^*K} \Lf^*$
defines a duality preserving functor $\pair{\Lf^*,\fh_K}$ between triangulated categories with weak dualities, from $(D(Y),\dual_K,\bid_K)$ to $(D(X),\dual_{\Lf^*K},\bid_{\Lf^*K})$, for any object $K$ of $D_{qc}(Y)$.

\begin{theo} \label{PullBack_theo}
Let $f:X \to Y$ be a morphism of schemes such that 
\begin{itemize}
\item[-] the objects $K$ and $\Lf^* K$ are dualizing. 
\item[-] $\Lf^*$ preserves $D_{b,c}$, 
\item[-] $\fh_K$ is an isomorphism on $D_{b,c}(Y)$,
\end{itemize}
Then $\pair{\Lf^*,\fh_K}$ induces a morphism on Witt groups
$$f^*: \W^i(Y,K) \to \W^i(X,\Lf^* K)$$
that we call pull-back. This pull-back therefore exists in particular if $K$ and $\Lf^*K$ are dualizing and
\begin{itemize}
\item[-] $f$ is of finite tor-dimension and $X,Y \in \Sch$ or 
\item[-] for any $f$ and $X, Y \in \Reg$ in which case
$K$ dualizing imples $\Lf^*K$ dualizing by Example \ref{examplesdual}.
\end{itemize}
\end{theo}
\begin{proof}
This follows from \cite[Theorem 4.1.2 and Lemma 2.2.6 (1)]{Calmes09}.
The Theorem of \loccit\ 
ensures the existence of the appropriate commutative 
diagrams in $D(X)$. The requirements in the Lemma of \loccit\  
that the dualities given by $K$ and $f^*K$ restrict as strong dualities to $D_{b,c}$ are satisfied by assumption, and the requirement that $\fh_K$ is an isomorphism when restricted
to $D_{b,c}$ follows from Proposition \ref{betaIsoFiniteTorDim_prop}.
\end{proof}

\begin{rema}\label{qcvariant_rema}
Note that we obtain the very same pull-back
when starting with the monoidal structure on $D_{qc}$ instead of $D$.
This follows from Proposition \ref{compareBetas_prop}.
\end{rema}
\begin{rema}
In \cite[Theorem 3.12]{Neeman10}, it is proved that if $f$ is an open immersion, then $\Lf^*K$ is automatically dualizing if $K$ is  dualizing.
\end{rema}

Let $X,Y \in \Sch$, let $K \in D_{qc}(Y)$ and let $f: X \to Y$ be a separated morphism. From \cite[Theorem 4.2.9]{Calmes09} applied to the closed monoidal category $D_{qc}(X)$, we obtain a morphism of functors $\rr_K: \Rf_* \dual'_{f^!K} \to \dual'_{K}\Rf_*$. By \loccit, the pair $\pair{\Rf_*,\rr_K}$ is duality preserving, \ie\ Diagram \eqref{dualPres_diag} commutes. 

\begin{theo} \label{PushForward_theo}
Let  $X,Y \in \Sch$ and $f:X \to Y$ be a separated morphism such that $\Rf_*$ preserves $D_{b,c}$. Let $K$ and $f^! K$ be dualizing. Then $\pair{\Rf_*,\rr_K}$ induces a morphisms of Witt groups
$$f_*: \W^i(X,f^! K) \to \W^i(Y,K)$$ 
that we call push-forward. This push-forward is therefore defined in particular if $f$
is proper and $K$ is dualizing (see Proposition \ref{properDualizing_prop}).
\end{theo}
\begin{proof}
This follows from \cite[Theorem 4.2.9 and Lemma 2.2.6. (1)]{Calmes09}.
For the Theorem of \loccit, consider 
the triangulated closed monoidal category $D_{qc}$. The fact that $\rr_K$ is an isomorphism 
follows from \cite[Prop. 4.3.3]{Calmes09} using Theorem 
\ref{projMorphIso_theo}. Then, apply the Lemma of \loccit\ to the subcategories $D_{b,c}$, to which the dualities restrict by definition of a dualizing object. 
Note that when $X$ and $Y$ are regular, the complete
proof works using directly $D_{b,c}$ as the triangulated closed monoidal category in \cite[Theorem 4.2.9]{Calmes09}. 
\end{proof}

\section{Properties} \label{properties_sec}

We now show that both push-forwards and pull-backs respect composition and that they commute in an appropriate way (``base change'') provided certain standard conditions hold. We also prove a projection formula for regular schemes.

\begin{theo}
For any $f:X \to Y$ and $g: Y \to Z$, 
\begin{enumerate}
\item there is an isomorphism $\Lf^* \circ \LL g^* \to \LL (g \circ f)^*$
between functors on $D(-)$ which is associative in the usual sense.
\item There is an isomorphim $\RR (g \circ f)_* \to \RR g_* \circ \RR f_*$ between functors on $D(-)$ which is associative in the usual sense, and respects the adjoint couple $(\LL(-)^*,\RR(-)_*)$ in the sense of \cite[Def. 5.1.5]{Calmes09}. 
\item When the schemes are separated and quasi-compact, and both $f$ and $g$ are separated, there is an isomorphism $f^! \circ g^! \to (g \circ f)^!$ between functors on $D_{qc}(-)$ which is associative in the usual sense, and which respects the adjoint couple $(\RR(-)_*,(-)^!)$ in the sense of \cite[Def. 5.1.5]{Calmes09}.
\end{enumerate}
\end{theo}
\begin{proof}
For the functors $\Lf^*$ on $D$, the isomorphism is in \cite[Theorem A (iii)]{Spaltenstein88} or \cite[3.6.4]{Lipman09}. For a proof that it is associative, see \cite[Scholium 3.6.10]{Lipman09}. The other points follow from the first one by \cite[Lemma 5.1.6]{Calmes09}.
\end{proof}

\begin{theo} \label{compoPullBack_theo}
The pull-back respects composition: the diagram
$$\xymatrix{
\W^i(Z,K) \ar[r]^-{g^*} \ar[drr]_{(gf)^*} & \W^i(Y,\LL g^* K) \ar[r]^-{f^*} & \W^i(X,\Lf^*\LL g^*K) \ar@{-}[d]^{\wr} \\
 & & \W^i(X, \LL (gf)^* K) 
}$$ 
commutes, under the conditions for the existence of the pull-backs $f^*$ and $g^*$ of Theorem \ref{PullBack_theo}.
\end{theo}
\begin{proof}
This follows from \cite[Theorem 5.1.3 and Cor. 5.1.4]{Calmes09} applied to the structures on $D(-)$.
\end{proof}

\begin{theo} \label{compoPushForward_theo}
The push-forward respects composition: the diagram 
$$\xymatrix{
\W^i(X,f^! g^! K) \ar@{-}[d]_{\wr} \ar[r]^-{f_*} & \W^i(Y,g^! K) \ar[r]^-{f^*} & \W^i(Z,K) \\
 \W^i(X, (gf)^! K) \ar[urr]_{(gf)_*}
}$$ 
commutes, under the conditions for the existence of the push-forward of Theorem \ref{PushForward_theo}.
\end{theo}
\begin{proof}
This follows from \cite[Theorem 5.1.9 and Cor. 5.1.10]{Calmes09} applied to the structures on $D_{qc}(-)$.
\end{proof}

We now prove a base change formula. Let us consider a pull-back diagram
$$\xymatrix{
V \ar[d]_{\bar{f}} \ar[r]^{\bar{g}} & Y \ar[d]^{f} \\
X \ar[r]_{g} & Z
}$$
By \cite[Section 5.2]{Calmes09}, 
we obtain a morphism of functors
$$\eps: \Lf^* \RR g_* \to \RR \bar{g}_* \LL \bar{f}^*$$
between functors on $D(X)$.

\begin{prop} \label{torIndep_prop}
If all schemes are in $\Sch$ and the diagram is tor-independent, \eg\ $f$ flat, the morphism $\eps$ is an isomorphism on $D_{qc}(X)$.
\end{prop}
\begin{proof}
The case where $f$ is flat 
is \cite[Proposition 3.9.5]{Lipman09} (all maps between schemes in $\Sch$
are ``concentrated" in the sense of \loccit). The more general case is \cite[Theorem 3.10.3]{Lipman09}. 
\end{proof}

Then, when the schemes are in $\Sch$, still by \cite[Section 5.2]{Calmes09}, $\eps$ induces a morphism
$$\gam: \LL \bar{f}^* g^! \to \bar{g}^! \Lf^*$$
between functors on $D_{qc}(X)$. It is an isomorphism on the subcategory $D_{-,qc}$ by 
\cite[Corollary 4.4.3]{Lipman09}. In particular, $\gam_K$ is an isomorphism when $K$ is dualizing (and thus in $D_{b,c}(Z)$).

\begin{theo}[Base change] \label{baseChange_theo}
Under the assumptions of Proposition \ref{torIndep_prop} and the ones for the pull-backs along $f$ and $\bar{f}$ and the push-forwards along $g$ and $\bar{g}$ to exist (Theorems \ref{PullBack_theo} and \ref{PushForward_theo}), the pull-back and push-forward satisfy a base change formula: the diagram
$$\xymatrix@R=4ex@C=4ex{
 & \W^i(V,\bar{g}^! \Lf^* K) \ar[r]^-{\bar{g}_*} & \W^i(Y, \Lf^* K) \\
\W^i(V,\LL\bar{f}^* g^!K) \ar[ur]_{\sim}^{\gam} & & \\
\W^i(X, g^! K) \ar[u]^{\bar{f}^*} \ar[rr]_-{g_*} & & \W^i(Z,K) \ar[uu]_{f^*}
}$$
commutes. 
\end{theo}
\begin{proof}
This follows from \cite[Theorem 5.2.1 and Corollary 5.2.2]{Calmes09}
applied to the structures on $D_{qc}(-)$, keeping in mind Remark 
\ref{qcvariant_rema}.
\end{proof}

We conclude this section with a projection formula for Witt groups, in the case of regular schemes. For this, we first need to introduce another natural morphism that will anyway be of some use even in the case of non regular schemes. 

When $f:X \to Y$ is a separated morphism in $\Sch$, using the functors $\Lf^*$ (monoidal), $\Rf_*$ and $f^!$ between the categories $D_{qc}(X)$ and $D_{qc}(Y)$, and the fact that the projection morphism $\q$ is an isomorphism by Theorem \ref{projMorphIso_theo}, we obtain a morphism of functors
$$\ssp: f^!(-) \LTens \Lf^*(-) \to f^!(- \LTens -)$$
by \cite[Proposition 4.3.1]{Calmes09}. 

\begin{lemm} \label{sspIso_lemm}
The morphism $\ssp: f^!A \LTens \Lf^*B \to f^!(A \LTens B)$ is an isomorphism when $B$ is a perfect complex.
\end{lemm}
\begin{proof}
The morphism $\ssp$ is compatible with open immersions by Diagram 41 of \cite[Proposition 5.2.5]{Calmes09}, and so we can restrict to the case of bounded complexes of vector bundles, then to vector bundles, then again using open immersions to the trivial bundle $\cO_Y$. In that case, one can show that $\ssp_{A,\cO_Y}$ coincides with the unit morphism of the monoidal structure $f^!(A)\LTens \cO_X \to f^!(A)$, and is therefore an isomorphism. By this coincidence we mean that the left diagram
$$\xymatrix@R=4ex{
f^! A \LTens \Lf^* \cO_Y \ar[r]^-{\ssp} & f^!(A \LTens \cO_Y) \\
f^! A \LTens \cO_X \ar@{-}[r]^-{\sim} \ar@{-}[u]^{\wr} & f^!A \ar@{-}[u]^{\wr}
}
\xymatrix@R=4ex{
\Rf_* B \LTens \cO_Y \ar[r]^-{\q} & \Rf_*(B \LTens \Lf^* \cO_Y) \\
\Rf_* B \ar@{-}[r]^-{\sim} \ar@{-}[u]^{\wr} & \Rf_* (B \LTens \cO_X) \ar@{-}[u]^{\wr}
}
$$ 
is commutative, in which the left vertical morphism is the fact that $\Lf^*$ is monoidal and in particular respects units, and the bottom and right maps are unit morphisms of the monoidal structures. By following the definition of $\ssp$ given in \cite[Proposition 4.3.1]{Calmes09} the commutativity of the left diagram follows from the one of the right hand side, which is in turn implied, using the definition of $\pi$ in \cite[Proposition 4.2.5]{Calmes09}, by the compatibility of the unit and monoidal structure morphisms for $\Lf^*$.
\end{proof}

For any scheme $X$ in $\Reg$, the derived tensor product preserves $D_{b,c}(X)$ (Theorem \ref{DXclosedsymm_theo} (3)). This gives two different products on Witt groups by the formalism of \cite{Gille03}, using \cite[Proposition 4.4.6 and Corollary 4.4.7]{Calmes09} applied to the closed monoidal structure of $D_{b,c}(X)$. We fix one of these products (say, the left one) for the following results. When $K$ and $L$ are shifted line bundles, thus dualizing complexes, the product is a pairing 
$$\W^i(X,K) \times \W^j(X,L) \to \W^{i+j}(X, K \otimes L).$$

\begin{theo}[Projection formula] \label{ProjFormula_theo}
For any proper morphism $f:X \to Y$
with $X,Y \in \Reg$, the pull-back and push-forward satisfy a projection formula: 
If $K,L$ are shifted line bundles on $Y$,
$x \in \W^i(X,f^!K)$ and $y \in \W^j(Y,L)$,
then $f_*(I(x.f^*y))=f_*(x).y$
in $\W^{i+j}(Y,K \otimes L)$ with $I$ the isomorphism from $\W^i(X,f^!K \otimes f^* L)$ to $\W^i(X,f^!(K \otimes L))$ induced by $\ssp_{K,L}$.
\end{theo}
\begin{proof}
First note that $L$ being a shifted line bundle explains the absence of derivations in the pull-back and tensor products above. Then, the morphism $\ssp_{K,L}$ is an 
isomorphism by Lemma \ref{sspIso_lemm}, 
thus the result follows from 
\cite[Theorem 5.5.1 and Corollary 5.5.2]{Calmes09} applied to the closed monoidal structure on $D_{b,c}$.
\end{proof}

\section{Reformulations in special cases} \label{reformulations_sec}

In this section, we give other canonical ways of writing the push-forward, under additional assumptions. 
\begin{nota}
For an equidimensional morphism $f: X \to Y$ of relative dimension $n$, let
$\rel_f$ denote the object $f^!(\cO_Y)[-n]:=T^{-n}f^!(\cO_Y)$.
\end{nota}
This notation is motivated by the fact that in several cases, this object can be identified with a geometric object called a relative dualizing sheaf and usually denoted $\rel_f$: see sections \ref{regClosedEmb_sec} and \ref{projSpace_sec} for the examples of regular embeddings and projective spaces.

\subsection{Relative dualizing sheaf} \label{relDual_sec}

\begin{theo} \label{PushForward2_theo}
Let $f:X \to Y$ be a proper morphism in $\Sch$, $K$ a dualizing complex on $Y$ such that $f^! K$ is a dualizing complex and assume that $\ssp_{\cO_Y,K}: f^! \cO_Y \LTens \Lf^* K \to f^!K$ is an isomorphism. Then we can rewrite the push-forward of Theorem \ref{PushForward_theo} as
$$f_*: \W^{i+d}(X,\rel_f \LTens \Lf^*K) \to \W^i(Y,K).$$ 
In particular, the hypotheses and therefore the conclusion hold if $K$ is dualizing and either of the two following conditions hold.
\begin{itemize}
\item[-] $f$ is quasi-perfect (see below, \eg\ of finite tor-dimension) and $f^! K$ is dualizing.
\item[-] $Y$ is a Gorenstein scheme, \eg\ regular.
\end{itemize}
\end{theo}
\begin{proof}
The reformulation of the push-forward is \cite[Definitions 6.1.3 and 6.1.4]{Calmes09}. When $f$ is quasi-perfect, \cite[Proposition 2.1]{Lipman07} shows that $\ssp$ is an isomorphism on all $D_{qc}(Y)$. Example 2.2 in \loccit\ shows that if $f$ is of finite tor-dimension, it is quasi-perfect. When $Y$ is Gorenstein, the only dualizing complexes are shifted line bundles, for which $\ssp$ is an isomorphism by Lemma \ref{sspIso_lemm}. 
\end{proof}

Let $g:Y \to Z$ be another proper morphisms in $\Sch$
and $M$ a dualizing complex on $Z$ 
and let 
$$ \iota_{f,g}: \rel_{f} \LTens \Lf^*(\rel_g \LTens \LL g^*M) \to \rel_{gf} \LTens \LL (gf)^* M$$
be the morphism defined in \cite[Theorem 6.1.5]{Calmes09}.
\begin{theo}
The push-forward of Theorem \ref{PushForward2_theo} respects composition: the morphism $\iota_{f,g}$ is an isomorphism and if $I$ denotes the isomorphism of Witt groups induced by $\iota_{f,g}$, then the push-forward on Witt groups defined above satisfies that $g_* f_* = (gf)_* I$.
\end{theo}
\begin{proof}
This follows from \cite[Theorem 6.1.5, Lemma 2.2.6 (2)]{Calmes09}. 
Note that $\iota_{f,g}$ is an isomorphism because it is a composition of isomorphisms under the assumptions for the reformulated push-forward to exist.
\end{proof}

\begin{theo}
In the situation of Theorem \ref{baseChange_theo} and under the assumptions of the reformulation of the push-forward above for the morphisms $g$ and $\bar{g}$, the base change theorem \ref{baseChange_theo} becomes
$$f^* g_* = \bar{g}_* I \bar{f}_*$$
where $I$ is the isomorphism of Witt groups induced by the isomorphism 
$$\iota: \LL \bar{f}^* (\rel_g \LTens \LL g^* K) \to \rel_{\bar{g}} \LTens \LL \bar{g}^* \Lf^* K.$$
\end{theo}
\begin{proof}
This follows from \cite[Theorem 6.1.7]{Calmes09}. Note that $\gam_{\cO_Z}$ is an isomorphism (see before Theorem \ref{baseChange_theo}).
\end{proof}

In the regular case, the projection formula \ref{ProjFormula_theo} becomes the following.
\begin{theo} \label{ProjFormula2_theo}
Let $f:X \to Y$ be a proper equidimensional morphism of relative dimension $d$ with $X,Y \in \Reg$. Then the push-forward of Theorem \ref{PushForward2_theo} and the pull-back of Theorem \ref{PullBack_theo} satisfy $f_* I (x.f^*(y))=f_*(x).y$ in $\W^{i+j}(Y,L \otimes K)$ for any $x \in \W^{i+d}(X,\rel_f \otimes f^* L)$ and $y\in \W^j(Y,K)$.
\end{theo}
\begin{proof}
See \cite[Theorem 6.1.9 and Corollary 6.1.10]{Calmes09}.
\end{proof}

\subsection{Smooth schemes over a base}

We now fix a base scheme $S\in \Sch$ with a dualizing complex $K_S$ and consider the category $\SmPr/S$ of schemes in $\Sch$ that are smooth, equidimensional and proper over $S$. For such a scheme $X$, let the structural morphism be denoted by $p_X: X \to S$ and its relative dimension over $S$ by $d_X$. Note that any separated morphism between schemes in $\SmPr/S$ is proper, being the composition of a closed embedding, its graph, and a proper projection.

\begin{nota}
Let $X\in \SmPr/S$. We set $\can_X=p_X^!(K_S)[-d_X]$.
Observe that $\can_X = \can_{p_X}$ if $K_S=\cO_S$.
\end{nota}

\begin{theo} \label{PushForward3_theo}
Let $f: X \to Y$ be a separated morphism, $X,Y \in \SmPr/S$ and let $L$ be a line bundle on $Y$. The push-forward can be written
$$f_*: \W^{i+d_X} (X,\can_X \otimes f^* L) \to \W^{i+d_Y}(Y, \can_Y \otimes L)$$
\end{theo}
\begin{proof}
First, let us note that when pulling back or tensoring by a line bundle, there is nothing to derive. This is why no $\LL$ appear in front of $f^*$ and $\otimes$. We then use Definitions \cite[Definitions 6.3.3 and 6.3.4]{Calmes09}. We need to check that the morphism
$$\can_X \otimes L \simeq f^! \can_Y \otimes f^* L \to f^!(\can_Y \otimes L)$$
is an isomorphism. This is the case by Lemma \ref{sspIso_lemm}. 
\end{proof}

\begin{theo} 
The push-forward of Theorem \ref{PushForward3_theo} respects composition.
\end{theo}
\begin{proof}
See \cite[Theorem 6.3.5 and Corollary 6.3.6]{Calmes09}.
\end{proof}

\begin{theo}
The push-forward of Theorem \ref{PushForward3_theo} satisfies flat base change.
\end{theo}
\begin{proof}
See \cite[Theorem 6.3.7 and Corollary 6.3.8]{Calmes09}.
\end{proof}

\section{Examples} \label{examples_sec}

Note that $f^!$ is unique up to
unique isomorphism whenever it is defined,
because it is always defined as a right adjoint
to $\Rf_*$. This allows us to use
computations of $f^!$ from 
\cite{Hartshorne66} and other sources
in the examples below.

\subsection{Finite field extensions}

The simplest example of a proper morphism is the case of a finite field extension $E/F$ giving rise to a finite morphism 
$$f:X=\Spec(E) \to \Spec(F)=Y.$$ 
The tilde construction gives equivalences of categories
$\Mod(F)\simeq\Qcoh(Y)$ and $\Mod(E)\simeq \Qcoh(X)$, and the subcategories of
finite dimensional vector spaces correspond to coherent sheaves of modules. We
thus describe all objects and functors through these equivalences of
categories. The only dualizing complex (up to shifts and isomorphisms) on $Y$
is $F$ itself. The functors $f^*=(-\otimes_F E)$ and $f_*= (-)|_F$ are exact,
there is nothing to derive. The functor $f^!$ is given 
by \cite[III §6]{Hartshorne66} as $\SHom_F(E,-)$ (mapping to $E$-vector spaces) and the unit and counit of the adjunction $(f_*,f^!)$ are respectively given by
$$\left\{\begin{array}{ccc} V & \to & \SHom_F(E, V|_F) \\ a & \mapsto & (e \mapsto e.a) \end{array}\right. \hspace{5ex} \left\{\begin{array}{ccc} \SHom_F(E,V')|_F & \to & V' \\ \phi & \mapsto & \phi(1) \end{array}\right.$$
for an $E$-vector space $V$ and an $F$-vector space $V'$.
For fields, the only nonzero Witt group modulo 4 is $\W^0$  
which is the classical Witt group of the field. 
So we are reduced to study push-forward for forms on vector spaces, \ie\ complexes concentrated in degree zero. 
Following the construction, it is easy to check that for any $E$-vector space $V$, the morphism $\rr: f_* \HHom{V,f^! F} \to \HHom{f_*V, F}$ coincides with the Cartan isomorphism of $F$-vector spaces
$$\SHom_E(V,\SHom_F(E,F))|_F \simeq \SHom_F(V|_F , F)$$
which sends a morphism $\phi: V \to \SHom_F(E,F)$ to the morphism $(a \mapsto \phi(a)(1))$. 
Thus, the push-forward $f_*: \W^0(E, \SHom_F(E,F)) \to \W^0(F)$ is a Scharlau
transfer (see \cite[p. 48]{Scharlau85}) with respect to the usual trace 
$Tr:E \to F$. To see this, note that
$Tr$ factors as $E \stackrel{\simeq}{\to} \SHom_F(E,F) \to F$ 
where the isomorphism is given by $e \mapsto (x \mapsto Tr(e.x))$
and $\SHom_F(E,F) \to F$ is the evaluation at $1$.

\subsection{Regular embeddings} \label{regClosedEmb_sec}

Let $F$ be a vector bundle of rank $d>0$ over $X$ with a regular section 
$s: \cO_X \to F$, \ie\ such that the corresponding
embedding $f:Z\subset X$ of the zero locus 
is a closed regular embedding of
codimension $d$. In that case the augmented Koszul resolution 
is exact \cite[IV, §2, Proposition 3.1]{Fulton85} 
and thus yields a quasi-isomorphism 
\begin{equation} \label{qis_eq}
\xymatrix@C=4ex{
\Kos_F \ar[d]_{\qis} \ar@{}[r]|{=} & (\; 0 \ar[r] & \Lambda^d F^\vee \ar[r] & \Lambda^{d-1} F^\vee \ar[r] & \cdots \ar[r] & F^\vee \ar[r] & \cO_X \ar[d] \ar[r] & 0 \; ) \\
f_* \cO_Z \ar@{}[r]|{=} & (\; & & & & 0 \ar[r] & f_* \cO_Z \ar[r] & 0 \;)
}
\end{equation}
from the Koszul complex $\Kos_F$ to $f_* \cO_Z$ concentrated in degree $0$.
Since $f$ is a closed embedding, thus finite, $f_*$ is exact and
coincides with $\Rf_*$. 
Let $\Delta_F= \Lambda^d F$ be the determinant of $F$. 
In this situation, we have $f^! A = f^*\Delta_F [-d] \otimes \Lf^* A$
for all $A \in D_{b,c}(X)$;
this may be extracted from \cite[III §7]{Hartshorne66}, see also
\cite[Proposition 1]{Verdier69}, after applying Lemma
\ref{sspIso_lemm} and using that $F$ is dual
to the cotangent sheaf.
By tensoring the augmented Koszul resolution with $\Delta_F$ and using the canonical isomorphisms $\Lambda^i F^\vee \otimes \Delta_F \cong \Lambda^{d-i} F$ and $f_* \cO_Z \otimes \Delta_F \cong f_* f^* \Delta_F$, we obtain the trace map $f_* f^! \cO_X \to \cO_X$ (counit of the adjunction $(\Rf_*,f^!)$) in the derived category as the composition of a usual map of complexes followed by the inverse of a quasi-isomorphism ($\cO_X$ is in degree $0$ and $f_*f^*\Delta_F$ in degree $-d$):
$$\xymatrix{
f_*f^! \cO_X \ar[dd] \ar@{}[r]|{=} & (\; 0 \ar[r] & 0 \ar[r] & \cdots \ar[r] & 0 \ar[r] & f_*f^* \Delta_F \ar[d]^{id} \ar[r] & 0 \; ) \\
 & (\; 0 \ar[r] & F \ar[r] & \cdots \ar[r] & \Delta_F \ar[r] & f_* f^* \Delta_F \ar[r] & 0 \; ) \\
\cO_X & (\; 0 \ar[r] & \cO_X \ar[u]_{s} \ar[r] & 0 \ar[r] & \cdots \ar[r] & 0 \ar[r] & 0 \; )
}$$

Now assume $Z$ is Gorenstein. Then $\cO_Z$ is dualizing, and the 
isomorphism $\cO_Z \to \HHom{\cO_Z,\cO_Z}$ adjoint to $\cO_Z \LTens \cO_Z
\simeq \cO_Z$ defines a form on $\cO_Z$, denoted by $\one_Z$. On the other
hand, there is a well-known form $\theta_F: \Kos_F \to \SHom(\Kos_F,
\Delta_F^\vee [d])$ (see \cite[§4]{Balmer05c})
 given by the canonical isomorphism $\Lambda^i F^\vee \simeq
 (\Lambda^{d-i}F^\vee)^\vee \otimes \Lambda^d F^\vee$ in degree $i$, with a
 sign chosen so that when $F=\oplus L_i$ is a direct sum of line bundles, 
this form is the tensor product of forms $\theta_{L_i}$ on Koszul complexes of length one
\begin{equation} \label{coneForm_eq}
\xymatrix{
\Kos_{L_i} \ar[d]_{\theta_{L_i}} \ar@{}[rr]|{=} & (\; 0 \ar[r] & L_i^\vee \ar[d]_{-1} \ar[r]^{s_i^\vee} & \cO_X \ar[d]^{1} \ar[r] & 0\;) \\
\SHom(\Kos_{L_i}, L_i^\vee [1]) \ar@{}[rr]|{=} & (\; 0 \ar[r] & L_i^\vee \ar[r]^{-s_i^\vee} & \cO_X \ar[r] & 0\;) \\
}
\end{equation}
representing elements in $\W^1(X,L_i^\vee)$.
The following proposition can be considered as a concrete description of the push-forward of $\one_Z$ along $f$.

\begin{prop} \label{regularEmbed_prop}
Let $Z$ and $X$ be Gorenstein schemes and $f: Z \to X$ be a closed regular embedding
of codimension $d$ defined as the zero locus of a regular section of a vector bundle $F$ of rank $d$ whose determinant $\Lambda^d F$ is denoted by $\Delta_F$. Then the image of the form $\one_Z: \cO_Z \stackrel{\simeq}{\to} \HHom{\cO_Z,\cO_Z}$ (adjoint to $\cO_Z \LTens \cO_Z \simeq \cO_Z$) under the composition
$$\xymatrix{\W^0(Z,\cO_Z) \ar@{}[r]|{\simeq} & \W^d(Z,f^! \Delta_F^\vee) \ar[r]^{f_*} & \W^d (X, \Delta_F^\vee)}$$
is a form $\phi$ such that the following diagram in $D(X)$ commutes.
$$\xymatrix{
\Kos_F \ar[d]_{\theta} \ar[r]^{\qis}_{\simeq} & f_* \cO_Z \ar[d]^{\phi} \\
\SHom(\Kos_F,\Delta_F^\vee [d]) \ar[r]_{\qis}^{\simeq} & \HHom{f_* \cO_Z, \Delta_F^\vee [d]} }$$
\end{prop}
\begin{proof}
Let $\delta_F: \Kos_F \otimes \Kos_F \to \Kos_F$ in $D(X)$ be the composition
$$\xymatrix@C=6ex{
\Kos_F \otimes \Kos_F \ar[r]^-{\qis \otimes \qis} & f_* \cO_Z \LTens f_* \cO_Z \ar[r]^{\fg} & f_* (\cO_Z\LTens \cO_Z) \ar[r]^-{\simeq} & f_* \cO_Z \ar[r]^{\qis^{-1}} & \Kos_F 
}$$
where $\fg$ is the morphism from \cite[Proposition 4.2.1]{Calmes09}. Note
that $\delta_F$ is in fact represented by a morphism of complexes (not just a
fraction): one can check that the map from $\Kos_F \otimes \Kos_F$ to $\Kos_F$
in degree $i$ is a sum of the canonical morphisms $\Lambda^k F \otimes
\Lambda^{i-k} F \to \Lambda^i F$ with appropriate signs.
We also consider the map $\sigma_F: \Kos_F \to \Delta_F^\vee[d]$ given by
$$\xymatrix@C=4ex{
\Kos_F \ar[d]_{\sigma_F} \ar@{}[r]|{=} & (\; 0 \ar[r] & \Lambda^d F^\vee \ar@{=}[d] \ar[r] & \Lambda^{d-1} F^\vee \ar[r] & \cdots \ar[r] & F^\vee \ar[r] & \cO_X \ar[r] & 0 \; ) \\
\Delta_F^\vee[d] \ar@{}[r]|{=} & (\; 0 \ar[r] & \Delta_F^\vee \ar[r] & 0 & & & & \;\;\;)
}$$
The following three lemmas together clearly imply the proposition.
\begin{lemm} \label{Kos_lemm}
The morphism of complexes $x_F: \Kos_F \to \SHom(\Kos_F,\Delta_F^\vee[d])$ defined as the composition
$$\xymatrix{
\Kos_F \ar[r] & \SHom(\Kos_F, \Kos_F \otimes \Kos_F) \ar[r]^-{(\delta_F)_*} & \SHom(\Kos_F,\Kos_F) \ar[r]^-{(\sigma_F)_*} & \SHom(\Kos_F,\Delta_F^\vee[d])
}$$
coincides with the form $\theta$ where the first map is the unit of the
adjunction of the tensor product and the internal Hom in the homotopy category.
\end{lemm}
\begin{lemm} \label{f_*OZ_lemm}
The form $\phi: f_* \cO_Z \to \HHom{f_* \cO_Z, \Delta_F^\vee [d]}$ coincides with the composition
$$\xymatrix@R=3ex{
f_* \cO_Z \ar[r] & \HHom{f_* \cO_Z, f_* \cO_Z \LTens f_* \cO_Z} \ar[r]^{\fg} & \HHom{f_*\cO_Z,f_* (\cO_Z \LTens \cO_Z)} \ar[d]^-{\wr} \\
\HHom{f_* \cO_Z, \Delta_F^\vee[d]} & \HHom{f_* \cO_Z, f_*f^! \Delta_F^\vee[d]} \ar[l] & \HHom{f_* \cO_Z, f_* \cO_Z} \ar[l]_-{\simeq} 
}$$
where the first map is the unit of the adjunction of the tensor product and the internal Hom in $D(X)$, the penultimate one is the identification $\cO_Z \simeq f^! \Delta_F^\vee[d]$ and the last one is induced by the counit of the adjunction $(f_*,f^!)$, \ie\ the trace map described above. 
\end{lemm}
\begin{lemm} \label{conclusion_lemm}
The composition of Lemma \ref{Kos_lemm} coincides with the one of Lemma \ref{f_*OZ_lemm} when $f_* \cO_Z$ is identified with $\Kos_F$ using $\qis$.
\end{lemm}
\begin{proof}[Proof of Lemma \ref{Kos_lemm}]
As we are dealing with honest morphisms of complexes
we may first reduce to open subsets on which $F$ is a sum of line bundles
$L_i$ (note that two morphisms in $D(X)$ are not necessarily equal
if they are equal when restricted to all open sets of an affine covering,
see for example \cite{Balmer07}). 
We then reduce to the case of codimension $d=1$, by multiplicativity of Koszul complexes: let $d=d_1+d_2$ and let $F=F_1 \oplus F_2$ where $F_1$ (resp. $F_2$) is the sum of the first $d_1$ line bundles (resp. last $d_2$). Let $f_1:\cO_{Z_1} \to X$ (resp. $f_2: \cO_{Z_2} \to X$) be the corresponding regular subschemes.
Then $\Delta_F \simeq \Delta_{F_1} \otimes \Delta_{F_2}$ and $\Kos_F \simeq \Kos_{F_1} \otimes \Kos_{F_2}$. We leave it to the reader to show that the diagram
$$\xymatrix@C=1ex{
\Kos_{F_1}\otimes \Kos_{F_2} \ar[d]_{x_{F_1} \otimes x_{F_2}} \ar[r]^{\simeq} & \Kos_F \ar[r]^-{x_F} & \SHom(\Kos_F,\Delta_F^\vee[d]) \ar[d]^{\simeq} \\ 
\SHom(\Kos_{F_1},\Delta_1^\vee[d_1])\otimes \SHom(\Kos_{F_2},\Delta_2^\vee[d_2]) \ar[rr]^{\dd} & & \SHom(\Kos_{F_1} \otimes \Kos_{F_2}, \Delta_1^\vee[d_1] \otimes \Delta_2^\vee[d_2]) 
}$$
commutes, where $\dd$ is the morphism defined as in \cite[Definition
4.4.1]{Calmes09}, using the monoidal structure on the homotopy
category. By definition, $\theta_F$, $\theta_{F_1}$ and $\theta_{F_2}$
make the same diagram commutative when they replace $x_F$, $x_{F_1}$ and
$x_{F_2}$. Hence it suffices to show the lemma for one line bundle $L$ and 
its associated Koszul complex of length one, which can be checked by hand. 
\end{proof}
\begin{proof}[Proof of Lemma \ref{f_*OZ_lemm}]
By definition, the form $\phi$ is given by the composition 
$$\xymatrix@C=5ex{
f_* \cO_Z \ar[r]^-{f_* \one_Z} & f_* \HHom{\cO_Z,\cO_Z} \ar[r]^-{\simeq} & f_*\HHom{\cO_Z,f^! \Delta_F^\vee[d]} \ar[r]^-{\rr} & \HHom{f_* \cO_Z, \Delta_F^\vee [d]}.
}$$ 
One proves using the closed monoidal structure that it coincides with the composition
$$\xymatrix{
f_* \cO_Z \ar[r]^-{\simeq} & f_* f^! \Delta_F^\vee[d] \ar[r]^-{} & f_* \HHom{\cO_Z,f^! \Delta_F^\vee[d]} \ar[r]^-{\rr} & \HHom{f_* \cO_Z, \Delta_F^\vee [d]}
}$$
where the second map is adjoint to the unit morphism of the monoidal structure. Then, looking back at the definition of $\rr$ and $\ff$ in \cite[Proposition 4.2.2 and Theorem 4.2.9]{Calmes09}, one sees that $\phi$ is the composition around the lower left corner of the commutative diagram

{\tiny
$$
\xymatrix@C=0ex{
f_* \cO_Z \ar[r] \ar[d] & \HHom{f_* \cO_Z, f_*\cO_Z \LTens f_* \cO_Z} \ar[d] \ar[r] & \HHom{f_* \cO_Z, f_*(\cO_Z \LTens \cO_Z)} \ar[d] \ar[r] & \HHom{f_* \cO_Z, f_* \cO_Z} \ar[d] \\
f_* f^! \Delta_F^\vee [d] \ar[d] \ar[r] & \HHom{f_* \cO_Z, f_* f^! \Delta_F^\vee[d] \LTens f_* \cO_Z} \ar[d] \ar[r] & \HHom{f_* \cO_Z,f_*(f^! \Delta_F^\vee[d] \LTens \cO_Z)} \ar[d] \ar[r] & \HHom{f_* \cO_Z, f_*f^! \Delta_F^\vee[d]} \ar[d] \\
f_* \HHom{\cO_Z,f^!\Delta_F^\vee[d]} \ar@/^3ex/[r] & \HHom{f_* \cO_Z, f_*\HHom{\cO_Z,f^!\Delta_F^\vee[d]}\LTens f_* \cO_Z} \ar@/^3ex/[r] & \HHom{f_* \cO_Z, f_*(\HHom{\cO_Z,f^!\Delta_F^\vee[d]}\LTens \cO_Z)} \ar[ur] & \HHom{f_* \cO_Z, \Delta_F^\vee[d]} 
}
$$
}

which thus proves the lemma (all squares in this diagram are commutative by obvious functorial reasons, and the triangle by adjunction).
\end{proof}
\begin{proof}[Proof of Lemma \ref{conclusion_lemm}]
This follows from the computation of the resolution of $f_* \cO_Z$ by $\Kos_F$ when computing the derived functors $\HHom{f_*\cO_Z,-}$ and $- \LTens f_* \cO_Z$.
\end{proof}
This finishes the proof of Proposition \ref{regularEmbed_prop}.
\end{proof}
\begin{rema}
If $F=F'\oplus L_1$, with $L_1$ a line bundle, $s=(s',s_1)$, $s'$ and $s_1$ transverse, the push-forward of $\one_Z$ is zero: decompose the inclusion $Z \subset X$ as $Z \subset Z(s') \subset X$ where $Z(s')$ is the zero locus of $s'$. Push-forwards respect composition and the push-forward of $\one_Z$ along $Z \subset Z(s')$ is already zero since it is the form \eqref{coneForm_eq} which is the cone of a (degenerate) form $s:L_1^\vee \to \cO_X$. \\
On the other hand, an example when this push-forward is nonzero can be
extracted from \cite{Balmer07_pre}. Let $k$ be a field and let $Gr_k(2,4)$ be
the Grassmannian of $2$-planes in $ \bbA_k^4$. A nonzero map from $\bbA_k^4$
to $\bbA_k$ induces a section of the dual $W^\vee$ of the universal subbundle
$W\subset \bbA_k^4$ of rank $2$. Its zero locus is a copy of $\bbP^2$
regularly embedded in $Gr(2,4)$. The push-forward of the unit form of $\bbP^2$
to $Gr(2,4)$ is nonzero by \cite{Balmer07_pre} 
where it is proved that it is an element of a basis of the total Witt group of $Gr(2,4)$ as a $\W(k)$-module.  
\end{rema}

\subsection{Projective spaces} \label{projSpace_sec}

Let $Y\in \Sch$ be a Gorenstein scheme, let $\cE$ be a vector bundle of rank
$r+1$ on $Y$ and let us examine when the unit form on $X=\bbP_Y(\cE)$ can be
pushed forward to $Y$ along $f:X \to Y$. Since $f$ is smooth (thus flat), we can use section \ref{relDual_sec}. 
In the case of a smooth morphism $f$, the object $\can_f$ of \ref{relDual_sec} is a line bundle, and it is the maximal exterior power of the relative cotangent bundle (see \cite[Ch. VII §4]{Hartshorne66}). 
Here, since $f$ is a projective bundle, it is given by $\can_f= f^*(\Delta_\cE)^\vee \otimes \cO(-r-1)$ (see \eg\ \cite[Appendix B.5.8]{Fulton98}).
If $r+1$ is even, we can push-forward the unit form $\one_X:\cO_X \simeq \HHom{\cO_X,\cO_X}$ from $\W^0(X,\cO_X)$ by using the composition
$$\W^0(X,\cO_X) \simeq \W^0(X,\cO(-r-1)) = \W^0(X,\can_f \otimes f^*(\Delta_\cE)) \to \W^{-r}(Y,\Delta_\cE)$$ 
where the first isomorphism is given by tensoring with the canonical form
$\phi_r=[\cO(-(r+1)/2) \stackrel{\simeq}{\to} \cO((r+1)/2)\otimes
\cO(-(r+1)) \cong \SHom_{\cO}(\cO(-(r+1)/2),\cO(-(r+1))]$
and the last map is the push-forward in the form of Theorem
\ref{PushForward2_theo}. Computing the image of $\one_X$ through this
composition means therefore computing the image of $\phi_r$ by the
push-forward. The complex on which $f_*(\phi_r)$ lives is $\Rf_*
(\cO(-(r+1)/2))$. But this complex is zero by \cite[2.1.15]{EGA3-1},
so $f_*(\phi_r)=0$. If $r+1$ is odd, there is no push-forward
induced by $f$ with source $\W^0(X,\cO_X)$ because then
there is no line bundle $K$ on $Y$ such that $\cO_X$ is equal to $f^*(\Delta_\cE^\vee) \otimes \cO(-r-1) \otimes f^*(K)$ up to a square in $Pic(Y)$.
In other words, pushing forward the unit form of
$\bbP_Y(\cE)$ to $Y$ is not very interesting: whenever it is possible, we get
zero. Of course, there are other forms on $\bbP^r(\cE)$ not mapping to
$0$ under the push-forward, as we will see in the following remark. 

\begin{rema}
Let us explain a potential source of confusion.
Let $i:\Spec k \to \bbP^r_k$ be
a rational point and $L$ a line bundle on $\bbP^r_k$. 
Since $\Pic(\Spec k)=0$, using first an isomorphism $\cO_k \simeq \rel_i \otimes i^*(L)$, we can push-forward from $\W^0(\Spec k, \cO_k)$ to $\W^r(\bbP^r_k,L)$ for any $L$. But for different $L$, we get {\em very different} push-forwards. Indeed, for example $\W^r(\bbP^r_k,\cO(-r))=0$ for odd $r$ (by \cite{Walter03_pre} or \cite{Balmer05b}) so any push-forward to there is obviously zero, whereas since $\can_k \simeq \cO_k$ the push-forward (written as in Theorem \ref{PushForward3_theo})
$$\W^0(\Spec k, \cO_k) \simeq \W^0(\Spec k, \can_k) \to \W^r(\bbP^r_k,\can_{\bbP^r_k})$$
is certainly nonzero, because we can further compose it by a push-forward back to $\W^0(\Spec k, \can_k)$ and since the push-forward respects composition, the composite is the identity. Note that this last case also gives an example of a form on $\bbP^r_k$ whose push-forward to $\Spec k$ is not zero.
More generally, this phenomenon of different push-forwards starting from the same group can happen whenever $f^*:\Pic(Y) \to \Pic(X)$ is not injective.
\end{rema}

\appendix

\section{q-flat and q-injective resolutions}

For the convenience of the reader, we include here well-known facts on q-flat
or q-injective objects, that are repeatedly used in the proofs of this
article. Most of them are due to Spaltenstein.

\begin{defi}
Let $X$ be a scheme and $A$ be an object in the homotopy category $K(X)$. We say that $A$ is q-flat (or K-flat) if the triangulated functor $(- \TTens A): K(X) \to K(X)$ preserves quasi-isomorphisms. We say that $A$ is q-injective (or K-injective) if the triangulated functor $\SHom_\bullet(-,A):K(X)^o \to K(X)$ preserves quasi-isomorphisms.
\end{defi}

\begin{exam}
A bounded above complex of flat $\cO_X$-modules is q-flat. A bounded below complex of injectives is q-injective.
\end{exam}

A discussion of q-flat and q-injective complexes can be found in \cite[Sections 1 and 5]{Spaltenstein88}. See in particular Propositions 1.5 and 5.3.

\begin{prop} \label{qflatinj_prop}
Let $A$ and $B$ be objects in $K(X)$ or $K(Y)$ and let $f:X \to Y$ be a morphism of schemes.
\begin{enumerate}
\item If $A$ and $B$ are q-flat, then so is $A \TTens B$. 
\item If $A$ is q-flat and $B$ is q-injective, then $\SHom_\bullet (A,B)$ is q-injective.
\item If $A$ is q-flat, then $f^* A$ is q-flat.
\end{enumerate}
\end{prop}
\begin{proof}
See \cite[Prop. 5.3 and 5.4]{Spaltenstein88}.
\end{proof}

The following two propositions summarize the equivalences of categories and
the properties of injectives that we need. Let $\Qcoh(X)$ denote the abelian
category of quasi-coherent sheaves on $X$ and
$\bVect{X}$ (resp. $\bCoh{X}$) the category of bounded complexes of locally free (resp. coherent) sheaves on $X$.
\begin{prop} \label{equivCat_prop}
Let $X \in \Sch$.
\begin{enumerate}
\item The natural functor $D(\Qcoh(X)) \to D_{qc}(X)$ is an equivalence of
  categories and thus the same is true for their homologically bounded,
  bounded below or bounded above subcategories and the subcategories
with coherent homology. 
\item The natural functor $D_b(\Coh(X)) \to D_{b,c}(\Qcoh(X))$ 
is an equivalence.
\item If $X\in \Reg$, then the natural functor $D(\bVect{X}) \to D_{b,c}(X)$ is an equivalence of categories.
\end{enumerate}
\end{prop}
\begin{proof}
Point (1) is \cite[Cor. 5.5]{Bokstedt93}.
In point (2), fully faithful follows
from \cite[Theorem 12.1]{Keller96}, second part: For affine schemes, use \cite[Example 12.3]{Keller96}. In general, take a finite affine cover of the noetherian scheme $X$ and then take the direct sum over the coherent sheaves on $X$ obtained by extending the coherent sheaves on the affine subschemes using \cite[I.6.9.7]{EGA1Springer}. 
Essentially surjective can be found
in \cite[Example 2.5.2]{Gille01_pre}
which follows from \cite[Lemme 2.1.2.c)]{Verdier96}
and an induction argument.
Point (3) can then be proved as follows. Let $\bCoh{X}$ be the category of bounded complexes of coherent sheaves. Decompose the functor $D(\bVect{X}) \to D_{b,c}(X)$ as
$$\xymatrix@C=2ex{D(\bVect{X}) \ar[r] & D(\bCoh{X}) \ar[r] & D_b(\Coh(X)) \ar[r] & D_{b,c}(\Qcoh(X)) \ar[r] & D_{b,c}(X).}$$
All these functors are equivalences of categories: the first one by the fact since $X$ is regular, every coherent sheaf has a finite resolution by locally free sheaves by \cite[§7 Point 1]{Quillen73}, the second one by \cite[Lemma 11.7]{Keller96}, the third one by Point (2) and the fourth by Point (1). 
\end{proof}

\begin{prop} \label{injectivesQcoh_prop}
Let $X \in \Sch$.
\begin{enumerate}
\item The category $\Qcoh(X)$ has enough injectives by \cite[B.3]{Thomason90}.
\item The natural inclusion $\Qcoh(X) \to \Mod(X)$ preserves injectives by \cite[B.4]{Thomason90}.
\item Every bounded below complex of quasi-coherent $\cO_X$-modules admits a quasi-isomorphism into a complex of $\Qcoh(X)$-injectives by (1) and \cite[I.4.6]{Hartshorne66}.
\item Every bounded below complex of $\Qcoh(X)$-injectives is q-injective in both $K(\Qcoh(X))$ and $K(X)$ by (2). 
\end{enumerate}
\end{prop}

\begin{coro} \label{qinjBounded_coro}
Let $X \in \Sch$. On objects in $D_{-,qc}(X)$ or $D_-(X)$, the unbounded right derived functors computed by q-injective resolutions (as in \cite{Spaltenstein88}) coincide with the more classical bounded below right derived functors computed by using resolutions by bounded below complexes of injectives (as in \cite{Hartshorne66}). 
\end{coro}
\begin{proof}
By the proposition, any object $A \in D_{-,qc}(X)$ is quasi-isomorphic to a bounded below complex of $\Qcoh(X)$-injectives which are also $\Mod(X)$-injective (resp. a bounded below complex of $\Mod(X)$-injectives) and this complex is q-injective. 
\end{proof}

Similarly, for q-flat resolutions, we have a weaker statement, sufficient for our purposes.

\begin{prop} \label{qflatBounded_prop}
Let $X \in \Sch$. On objects in $D_{+,qc}(X)$ or $D_+(X)$, the unbounded left derived functors computed using q-flat resolutions can be computed by using bounded above resolutions consisting of flat $\cO_X$-modules.
\end{prop}
\begin{proof}
This follows from the fact that any $\cO_X$-module is a quotient of a flat one (see \cite[2.5.5]{Lipman09}). 
\end{proof}

\bibliography{PushforwardWitt}

\end{document}